\documentclass[12pt]{article}
\usepackage[spanish,english]{babel}
\usepackage[utf8]{inputenc}

\usepackage{amsmath}
\usepackage{amsfonts}
\usepackage{latexsym}
\usepackage{epsfig}
\usepackage{graphicx}
\usepackage{url}

\usepackage{amssymb}
\usepackage{multirow}

\newtheorem{theorem}{Theorem}

\newtheorem{corollary}[theorem]{Corollary}

\newtheorem{lemma}[theorem]{Lemma}

\newtheorem{proposition}[theorem]{Proposition}

\newenvironment{proof}[1][Proof]{\textbf{#1.} }{\ \rule{0.5em}{0.5em}}

 \usepackage[colorlinks=true]{hyperref}
 \usepackage{amsmath,amssymb,eucal}
 \usepackage[spanish]{babel}
 \usepackage[pdftex]{color}

\title{Some Upper Bounds on \\ Ramsey Numbers Involving $C_4$}

\author{Luis Boza\\
            {\small Departamento de Matem\'atica Aplicada I, Universidad de Sevilla, boza@us.es}\\
	\and
          Stanis\l{}aw Radziszowski\\
          {\small Department of Computer Science, Rochester Institute of Technology, spr@cs.rit.edu}}

\begin{document}
\maketitle

\begin{abstract}
We obtain some new upper bounds on the Ramsey numbers of the form
$R(\underbrace{C_4,\ldots,C_4}_m,G_1,\ldots,G_n)$,
where $m\ge 1$ and $G_1,\ldots,G_n$ are arbitrary graphs.
We focus on the cases of $G_i$'s being complete, star $K_{1,k}$
or book graphs $B_k$, where $B_k=K_2+kK_1$. If $k\ge 2$, then our main upper bound theorem implies that
$$R(C_4,B_k) \le R(C_4,K_{1,k})+\left\lceil\sqrt{R(C_4,K_{1,k})}\right\rceil+1.$$

Our techniques are used to obtain new upper bounds in several
concrete cases, including:
$R(C_4,K_{11})\leq 43$, $R(C_4,K_{12})\leq 51$, $R(C_4,K_3,K_4)\leq 29$, $R(C_4,K_4,K_4)\leq 66$, $R(C_4,K_3,K_3,K_3)\leq 57$, $R(C_4,C_4,K_3,K_4)\leq 75$, and $R(C_4,C_4,K_4,K_4)\leq 177$, and also $R(C_4,B_{17})\leq 28$.

\end{abstract}

\section{Introduction}

For $n$ given graphs $H_1, H_2, \ldots, H_n$, the Ramsey number $R(H_1,H_2,\ldots,H_n)$ is the smallest integer $R$ such that if we arbitrarily color the edges of a complete graph of order $R$ with $n$ colors, then it contains a monochromatic copy of $H_i$ in color $i$, for some $1\le i\le n$.

We will use the following notation from \cite{R}: $K_k$ is a complete graph on $k$ vertices, the graph $kG$ is formed by $k$ disjoint copies of $G$, $G\cup H$ stands for vertex disjoint union of graphs, and the join graph $G+H$ is obtained by adding all of the edges between vertices of $G$ and $H$ to $G\cup H$. $C_k$ is a cycle on $k$ vertices, $P_k$ is a path on $k$ vertices, $K_{1,k}=K_1+kK_1$ is a star on $k+1$ vertices, and $B_k=K_2+kK_1$ is a book on $k+2$ vertices.

An $(H_1,\ldots,H_n)$-coloring of the edges of $K_N$ is a coloring using $n$ colors, such that it does not contain any monochromatic copy of $H_i$ in color $i$, for any $i$, $1 \le i \le n$. Note that if such coloring exists, then $N<R(H_1,\ldots,H_n)$. In the case of 2 colors, we will interpret graphs $G$ as colorings in which the edges of $G$ are assigned the first color, and the nonedges are assigned the second color.

Let $G$ be a graph or a coloring of edges, and let $V(G)$ denote the vertex set of $G$. For $v\in V(G)$, $G-v$ is the graph or the coloring induced by $V(G)\setminus\{v\}$. If $G$ is a coloring using $n$ colors and $v\in V(G)$, then $d_i(v)$ is the number of edges in color $i$ incident to $v$ in $G$. If $G$ is an $(H_1,\ldots,H_n)$-coloring, $1\le i\le n$, $v\in
 V(G)$ and $u_i\in V(H_i)$, then an elementary property of Ramsey colorings implies that $d_i(v)\le R(H_1,\ldots,H_{i-1},H_i-u_i,H_{i+1},\ldots,H_n)-1$. Numerous results on 2-color and multicolor Ramsey numbers involving $C_4$ are summarized in the dynamic survey \cite{R}, mainly in sections 3.3 (note that $C_4=K_{2,2}$), 4, and 6.[4,5,6,7].

The main goal of this paper is the derivation of some new upper bounds on the Ramsey numbers of the form
$R(\underbrace{C_4,\ldots,C_4}_m,G_1,\ldots,G_n)$,
where $m\ge 1$ and $G_1,\ldots,G_n$ are arbitrary graphs. The main result,
Theorem \ref{MT}, is obtained in Section 2. Then, in Sections 3 and 4 we
focus on the cases of $G_i$'s being complete, star or book graphs. Also in these sections several new concrete upper bounds are presented.

\section{Main result}

The main objective of this section is to obtain Theorem \ref{MT} claiming a new upper bound on the Ramsey numbers of the form $R(\underbrace{C_4,\ldots,C_4}_m,G_1,\ldots,G_n)$, with only relatively mild technical constraints. We need some auxilliary results, which will be presented first.

\begin{lemma} Sedrakyan's inequality \cite{Eng}.

For any real numbers $a_1,\ldots,a_m$ and positive real numbers $b_1,\ldots,b_m$, we have

$$\sum_{k=1}^m \dfrac{a_k^2}{b_k^2}\ge \left(\dfrac{\sum_{k=1}^m a_k}{\sum_{k=1}^m b_k}\right)^2.$$
\end{lemma}

Note that if $b_k=1$ for all $k$, $1 \le k \le m$, then Lemma 1 reduces to:

\begin{corollary} \label{se}
$\sum_{k=1}^m a_k^2\ge \dfrac{(\sum_{k=1}^m a_k)^2}{m}.$
\end{corollary}

A simple argument, involving just the basic definition of Ramsey numbers, leads to the next lemma.
\begin{lemma} \label{l1}
$$R(P_3,H_1,\ldots,H_n)+1\le R(C_4,H_1\cup K_1,\ldots,H_n\cup K_1)=$$
$$\max\{R(C_4,H_1,\ldots,H_n),|V(H_1)|+1,\ldots,|V(H_n)|+1\}.$$
\end{lemma}
\begin{proof}
Let
$N=R(P_3,H_1,\ldots,H_n)-1$. Consider any $(P_3,H_1,\ldots,H_n)$-coloring of $K_N$. By adding a new vertex adjacent to all of $K_N$ and using the first color for the new edges, a $(C_4,H_1\cup K_1,\ldots,H_n\cup K_1)$-coloring of $K_{N+1}$ is obtained. Thus, $N+1<R(C_4,H_1\cup K_1,\ldots,H_n\cup K_1)$ and the first part of the lemma is obtained. Next, observe that any graph $G$ containing $H_n$ contains $H_n\cup K_1$ as well, if $|V(G)|>|V(H_n)|$. Thus, $R(F,H_1,\ldots,H_{n-1},H_n\cup K_1)=$ $\max\{R(F,H_1,\ldots,H_{n-1},H_n),$
$|V(H_n)|+1\}$. We complete the proof by using the same argument for all colors.
\end{proof}

\begin{lemma} \label{l2}
Let $m\geq 1$ and $n\geq 0$. Consider $n$ graphs, $G_1,\ldots,G_n$. For each color $i$ with $1\leq i\leq n$, let $G'_i=G_i-w_i$, where $w_i\in V(G_i)$, and
let $r_i$'s be integers such that $$r_i\ge R(P_3,\underbrace{C_4,\ldots,C_4}_{m-1},G_1,\ldots,G_{i-1},G'_i,G_{i+1},\ldots,G_n).$$
Let $R=R(P_3,\underbrace{C_4,\ldots,C_4}_{m-1},G_1,\ldots,G_n)$. Then, we have
$$R\le \sum_{i=1}^n r_i-n+3+\dfrac{m^2-m}{2}+\left\lfloor\sqrt{\dfrac{\left(m^2-m\right)^2}{4}+(m-1)^2\left(\sum_{i=1}^n r_i-n+1\right)}\right\rfloor.\hspace{.8cm}(1)$$
\end{lemma}
\begin{proof}
Let $N=R-1$ and $G$ be a $(P_3,\underbrace{C_4,\ldots,C_4}_{m-1},G_1,\ldots,G_n)$-coloring of the edges of $K_N$. Let $v_0\in V(G)$ such that
$\sum_{i=2}^md_i(v_0)=\min_{v\in V(G)}\{\sum_{i=2}^m d_i(v)\}$.

In order to avoid a $P_3$ of the first color, we have $d_1(v_0)\leq 1$. If $1\leq i\leq n$, in order to prevent a $G_i$ of color $i+m$, we need $d_{i+m}(v_0)\leq r_i-1$. Hence, we arrive at the relation
$$N=1+\sum_{i=1}^{m+n} d_i(v_0)\le 2+\sum_{i=2}^m d_i(v_0)+\sum_{i=1}^n (r_i-1)=2-n+\sum_{i=2}^{m} d_i(v_0)+\sum_{i=1}^n r_i.\mbox{ \hspace{.5cm} } (2)$$

If $m=1$, then $R=N+1\leq 3-n+\sum_{i=1}^n r_i$, and the result is obtained.

Now, let us assume that $m\geq 2$.
For each color $i\in\{2,\ldots,m\}$, since there is no $C_4$ of color $i$, for any pair of vertices $u,v\in V(G)$ there is at most one vertex connected to both $u$ and $v$ by edges of color $i$. Therefore,
$\sum_{v\in V(G)}\left(\begin{array}{c}d_i(v)\\2\end{array}\right)\le\left(\begin{array}{c}N\\2\end{array}\right)$, and
$$\sum_{v\in V(G)}\left(\sum_{i=2}^m d_i(v)^2-\sum_{i=2}^m d_i(v)\right)=\sum_{i=2}^m\sum_{v\in V(G)}d_i(v)(d_i(v)-1)\leq (m-1)N(N-1).$$
Then, by Corollary \ref{se}, for any $v\in V(G)$ we have $\sum_{i=2}^m d_i(v)^2\ge\dfrac{(\sum_{i=2}^m d_i(v))^2}{m-1}$,

\noindent
and thus
$$(m-1)N(N-1)\ge \sum_{v\in V(G)}\left(\sum_{i=2}^m d_i(v)^2-\sum_{i=2}^m d_i(v)\right)\ge$$
$$\sum_{v\in V(G)}\left(\dfrac{(\sum_{i=2}^m d_i(v))^2}{m-1}-\sum_{i=2}^m d_i(v)\right)=\sum_{v\in V(G)}\left(\sum_{i=2}^m d_i(v)\right)\left(\dfrac{\sum_{i=2}^m d_i(v)}{m-1}-1\right)\ge$$
$$N\sum_{i=2}^m d_i(v_0)\left(\dfrac{\sum_{i=2}^m d_i(v_0)}{m-1}-1\right)=N\left(\dfrac{(\sum_{i=2}^m d_i(v_0))^2}{m-1}-\sum_{i=2}^m d_i(v_0)\right).$$

\noindent
Hence, using (2), we obtain
$$\dfrac{(\sum_{i=2}^m d_i(v_0))^2}{m-1}-\sum_{i=2}^m d_i(v_0)\leq (m-1)(N-1)\le$$

$$(m-1)\left(1-n+\sum_{i=2}^{m} d_i(v_0)+\sum_{i=1}^n r_i\right)$$

\noindent
and
$$(\sum_{i=2}^m d_i(v_0))^2-(m-1)\sum_{i=2}^m d_i(v_0)\le(m-1)^2\left(1-n+\sum_{i=2}^{m} d_i(v_0)+\sum_{i=1}^n r_i\right),$$

\noindent
which implies
$$\left(\sum_{i=2}^m d_i(v_0)\right)^2-m(m-1)\sum_{i=2}^m d_i(v_0)-(m-1)^2\left(1-n+\sum_{i=1}^n r_i\right)\le 0.$$

Consequently, seeing the latter as a quadratic in $\sum_{i=2}^m d_i(v_0)$, we have that

$$\sum_{i=2}^m d_i(v_0)\le \dfrac{m^2-m}{2}+\sqrt{\dfrac{\left(m^2-m\right)^2}{4}+(m-1)^2\left(\sum_{i=1}^n r_i-n+1\right)}.$$

Thus, by (2),
$$R\le \sum_{i=1}^n r_i-n+3+\dfrac{m^2-m}{2}+\sqrt{\dfrac{\left(m^2-m\right)^2}{4}+(m-1)^2\left(\sum_{i=1}^n r_i-n+1\right)}.$$

Since $R$ is an integer, the result is obtained.
\end{proof}

Using Lemmas \ref{l1} and \ref{l2}, we obtain the next (and last) lemma.

\begin{lemma} \label{P3}
Let $m\ge 1$ and $n\ge 0$. Consider any graphs $G_1,\ldots,G_n$. For each color $i$, $1\leq i\leq n$, let $G'_i=G_i-w_i$, where $w_i\in V(G_i)$, and let $r_i$'s be integers such that
$$r_i\geq R(\underbrace{C_4,\ldots,C_4}_m,G_1,\ldots,G_{i-1},G'_i,G_{i+1},\ldots,G_n).$$
Assume further that
$R(\underbrace{C_4,\ldots,C_4}_m,G_1,\ldots,G_n)>\max_{1\le i\le n}\{|V(G_i)|\}$, and also
that when $m=1$ then $G_i\ne K_2$ for some $i\in\{1,\ldots,n\}$.
Then we have
$$R(P_3,\underbrace{C_4,\ldots,C_4}_{m-1},G_1,\ldots,G_n)\leq\sum_{i=1}^n r_i-n+\frac{m^2+m}{2}+\left\lceil m\sqrt{\frac{(m+1)^2}{4}+\sum_{i=1}^n r_i-n}\right\rceil.\hspace{.3cm}(3)$$
\end{lemma}
\begin{proof}
Let $RHS(1)$ denote the right-hand side of inequality (1) in Lemma \ref{l2}, and let $RHS(3)$ denote the right-hand side of inequality (3).
In order to prove this lemma, by Lemma \ref{l2}, it suffices to show that
$RHS(3) \ge RHS(1)$. In the proof below, among other steps, we will use an easy observation that for any positive integer $k$, it is true that $\left\lceil\sqrt{k+1}\right\rceil = \left\lfloor\sqrt{k}\right\rfloor+1$.

If $m\ge 2$ then $RHS(3)=$
$$\sum_{i=1}^n r_i-n+\frac{m^2+m}{2}+1+\left\lfloor \sqrt{\frac{(m^2+m)^2}{4}+m^2\left(\sum_{i=1}^n r_i-n\right)-1}\right\rfloor=$$
$$\sum_{i=1}^n r_i-n+1+m+\frac{m^2-m}{2}+\left\lfloor \sqrt{\frac{(m^2-m)^2}{4}+m^3+m^2\left(\sum_{i=1}^n r_i-n\right)-1}\right\rfloor=$$
$$\sum_{i=1}^n r_i-n+1+m+\frac{m^2-m}{2}+\left\lfloor \sqrt{\frac{(m^2-m)^2}{4}+m^3+(m-1)^2\left(\sum_{i=1}^n r_i-n+1\right)-(m-1)^2-1}\right\rfloor$$

$\ge RHS(1).$

\bigskip
If $m=1$, let $i_0$ be such that $G_{i_0}\ne K_2$, so that $r_{i_0}\ge 2$ and $\sum_{i=1}^n r_i-n\ge 1$. Then
$$RHS(3)=\sum_{i=1}^n r_i-n+1+\left\lceil \sqrt{1+\sum_{i=1}^n r_i-n}\right\rceil\ge \sum_{i=1}^n r_i-n+3=RHS(1),$$

where in the latter the $RHS$'s were simplified using $m=1$.
\end{proof}

Now, we are ready to present our main result:

\begin{theorem} \label{MT}
Let $m\ge 1$ and $n\ge 0$. Consider $n$ graphs, $G_1,\ldots,G_n$. For each color $i$ with $1\leq i\leq n$, let $G'_i=G_i-w_i$, where $w_i\in V(G_i)$, and let $r_i$'s be integers such that
$$r_i\geq R(\underbrace{C_4,\ldots,C_4}_m,G_1,\ldots,G_{i-1},G'_i,G_{i+1},\ldots,G_n).$$
Assume further that
$R=R(\underbrace{C_4,\ldots,C_4}_m,G_1,\ldots,G_n)>\max_{1\le i\le n}\{|V(G_i)|\}$, and also
that when $m=1$ then $G_i\ne K_2$ for some $i\in\{1,\ldots,n\}$.
Then, we have
$$R \leq \sum_{i=1}^n r_i-n+1+\frac{m^2+m}{2}+\left\lceil m\sqrt{\frac{\left(m+1\right)^2}{4}+\sum_{i=1}^n r_i-n}\right\rceil.$$
\end{theorem}
\begin{proof}
Set $N=R-1$, and let $G$ be a $(\underbrace{C_4,\ldots,C_4}_m,G_1,\ldots,G_n)$-coloring of the edges of $K_N$. Let $v_0\in V(G)$ such that
$\sum_{i=1}^md_i(v_0)=\min_{v\in V(G)}\{\sum_{i=1}^m d_i(v)\}$. For $1\leq i\leq n$, in order to avoid $G_i$ of color $i+m$, we must have $d_{i+m}(v_0)\leq r_i-1$. Hence, we also have $$N=1+\sum_{i=1}^{m+n} d_i(v_0)\leq 1-n+\sum_{i=1}^{m} d_i(v_0)+\sum_{i=1}^n r_i.\mbox{\hspace{3cm}(4)}$$

For each $i\in\{1,\ldots,m\}$, the number of $P_3$'s in color $i$ cannot exceed $\left(\begin{array}{c}N\\2\end{array}\right)$, since otherwise they would force a $C_4$ in color $i$. Thus,
$\sum_{v\in V(G)}\left(\begin{array}{c}d_i(v)\\2\end{array}\right)\le\left(\begin{array}{c}N\\2\end{array}\right)$.
If $\sum_{v\in V(G)}\left(\begin{array}{c}d_1(v)\\2\end{array}\right)=\left(\begin{array}{c}N\\2\end{array}\right)$, then by the Friendship Theorem \cite{ERF}, which states that in
any graph in which any two vertices have precisely one common neighbor, then there is a vertex which is adjacent to all other vertices.
In that case, let $u$ be the vertex adjacent to all the others with edges of the first color. $G-u$ is a $(P_3,\underbrace{C_4,\ldots,C_4}_{m-1},G_1,\ldots,G_n)$-coloring of $K_{N-1}$, so $R-2=N-1\geq R(P_3,\underbrace{C_4,\ldots,C_4}_{m-1},G_1,\ldots,G_n)-1$, and by Lemma \ref{P3}, the result follows.

Similarly, the same argument applies if $\sum_{v\in V(G)}\left(\begin{array}{c}d_i(v)\\2\end{array}\right)=\left(\begin{array}{c}N\\2\end{array}\right)$ for some $i\leq m$. Therefore, we can assume that $\sum_{v\in V(G)}d_i(v)(d_i(v)-1)< N(N-1)$ for all $i$ and
$$\sum_{v\in V(G)}\left(\sum_{i=1}^m d_i(v)^2-\sum_{i=1}^m d_i(v)\right)=\sum_{i=1}^m\sum_{v\in V(G)}d_i(v)(d_i(v)-1)< mN(N-1).$$

Then, by Corollary \ref{se}, for any $v\in V(G)$ we have $m\sum_{i=1}^m d_i(v)^2\ge(\sum_{i=1}^m d_i(v))^2$, and further
$$mN(N-1)>\sum_{v\in V(G)}\left(\sum_{i=1}^m d_i(v)^2-\sum_{i=1}^m d_i(v)\right)\ge$$
$$\sum_{v\in V(G)}\left(\dfrac{(\sum_{i=1}^m d_i(v))^2}{m}-\sum_{i=1}^m d_i(v)\right)=\sum_{v\in V(G)}\left(\sum_{i=1}^m d_i(v)\right)\left(\dfrac{\sum_{i=1}^m d_i(v)}{m}-1\right)\ge$$
$$N\sum_{i=1}^m d_i(v_0)\left(\dfrac{\sum_{i=1}^m d_i(v_0)}{m}-1\right)=N\left(\dfrac{(\sum_{i=1}^m d_i(v_0))^2}{m}-\sum_{i=1}^m d_i(v_0)\right).$$
Therefore, by (4), we see that $$\dfrac{(\sum_{i=1}^m d_i(v_0))^2}{m}-\sum_{i=1}^m d_i(v_0)< m(N-1)\leq m\left(-n+\sum_{i=1}^{m} d_i(v_0)+\sum_{i=1}^n r_i\right)$$
and
$$(\sum_{i=1}^m d_i(v_0))^2-m(m+1)\sum_{i=1}^m d_i(v_0)-m^2\left(-n+\sum_{i=1}^n r_i\right) < 0,$$
and hence
$$\sum_{i=1}^m d_i(v_0)\leq \dfrac{m^2+m}{2}+\sqrt{\dfrac{\left(m^2+m\right)^2}{4}+m^2\left(\sum_{i=1}^n r_i-n\right)-1}.$$

Consequently, by (4), $$R=N+1\leq 2+\sum_{i=1}^n r_i-n+\dfrac{m^2+m}{2}+\sqrt{\dfrac{\left(m^2+m\right)^2}{4}+m^2\left(\sum_{i=1}^n r_i-n\right)-1}.$$

Since $R$ is an integer, we have
$$R\le 2+\sum_{i=1}^n r_i-n+\dfrac{m^2+m}{2}+\left\lfloor\sqrt{\dfrac{\left(m^2+m\right)^2}{4}+m^2\left(\sum_{i=1}^n r_i-n\right)-1}\right\rfloor$$
$$=\sum_{i=1}^n r_i-n+1+\frac{m^2+m}{2}+\left\lceil m\sqrt{\frac{\left(m+1\right)^2}{4}+\sum_{i=1}^n r_i-n}\right\rceil,$$
and the result follows.
\end{proof}

Note that if $m\ge 2$ and $n=0$, then the bound in Theorem \ref{MT} coincides with the known result $R(\underbrace{C_4,\ldots,C_4}_m)\le m^2+m+1$ \cite{Chu2,Ir}.

\section{Complete graphs}

In this section, we focus attention on concrete upper bounds
for the Ramsey numbers of the form
$R(\underbrace{C_4,\ldots,C_4}_{m},G_1,\ldots,G_n)$,
where all $G_i$'s are complete graphs, for $1 \le i \le n$.
We gather our results in Table \ref{table}, in which
the new upper bounds are shown in the last column.
Proposition \ref{pro29} below provides the upper
bound in Case {\#}3, while all other cases are derived in the proof of Theorem \ref{cases}.

\bigskip
\begin{table}[h]
\begin{center}
\begin{tabular}{|c|l|c|l|l|c|}
\hline
Case &
\multirow{2}{*}{Ramsey number} & \multirow{2}{*}{$m,n,\sum_{i=1}^n r_i$} &
lower & old upper & new upper \\
\#&& & bound & bound & bound\\ \hline
1& $R(C_4,K_{11})$ & $1,1,36$ & 40 \cite{VO} & 44\ \ \cite{LaLR} & 43 \\
2& $R(C_4,K_{12})$ & $1,1,43$ & 43 (*) & 52\ \ \cite{LaLR} & 51 \\
3& $R(C_4,K_3,K_4)$ & Proposition \ref{pro29} & 27\ \cite{DyDz1} & 32\ \ \cite{XSR1} & 29 \\
4& $R(C_4,K_4,K_4)$ & $1,2,58$ & 52 \cite{XSR1} & 71\ \ \cite{LidP} & 66 \\
5& $R(C_4,K_3,K_3,K_3)$ & $1,3,51$ & 49 \cite{Bev} & 59\ \ \cite{LidP} & 57 \\
6& $R(C_4,C_4,K_3,K_4)$ & $2,2,57$ & 43 \cite{DyDz1} & 76 \ \cite{XSR1} & 75 \\
7& $R(C_4,C_4,K_4,K_4)$ & $2,2,150$ & 87 \cite{XSR1} & 179 \cite{XSR1} & 177 \\ \hline
\end{tabular}
\caption{New bounds on Ramsey numbers of $C_4$ versus complete graphs described in Section 3:
parameters, lower bounds and old and new upper bounds. (*) Lower bound 43 in case {\#}2 is easily obtained by adding vertex-disjoint $K_3$ to the lower bound witness graph in case {\#1}.
In all cases, except case {\#}3,
the new upper bound is obtained by using Theorem \ref{MT}.}
\label{table}
\end{center}
\end{table}

\begin{proposition} \label{pro29}
$R(C_4, K_3, K_4) \le 29$.
\end{proposition}
\begin{proof}
First, we note that $R(K_3,K_4)=9$ \cite{GG} and $R(C_4,K_9)=30$ \cite{LaLR}.
Hence, if there exists any $(C_4, K_3, K_4)$-coloring $G$ of $K_{29}$,
then by merging the last two colors of $G$ we obtain a $(C_4, K_9)$-coloring,
i.e., a $C_4$-free graph $G'$ on 29 vertices with maximum independent
set of order at most 8. All such graphs were obtained in \cite{LaLR},
and up to isomorphism there are 267 of them.

We verified by computations
that for every such graph (one of 267 possible graphs), its non-edges
cannot be partitioned into a $K_3$-free graph and a $K_4$-free graph.
Thus, no $(C_4, K_3, K_4)$-coloring of $K_{29}$ exists, and
the bound in the proposition holds.
\end{proof}

\begin{theorem} \label{cases}
The upper bounds in the last column of Table \ref{table} hold.
\end{theorem}

\begin{proof}
Proposition \ref{pro29} proves the bound in Case {\#3}.
The upper bounds in
all other cases are obtained by applying Theorem \ref{MT} with
some additional simple steps, as described below.
\begin{itemize}
\item[\bf \#1.]
It is known that $R(C_4,K_{10})=36$ \cite{LaLR}.
Theorem \ref{MT} with $m=n=1$, $G_1=K_{11}$ and $r_1=36$ gives
$R(C_4,K_{11})\le 36+1+\sqrt{36}=43$.

\item[\bf \#2.]
Let $r_1=43$, so by Case \#1 we have $r_1 \ge R(C_4,K_{11})$.
With $m=n=1$ and $G_1=K_{12}$, we obtain
$R(C_4,K_{12}) \le r_1+1+\left\lceil\sqrt{r_1}\right\rceil=44+7=51$.

\item[\bf \#4.]
Let $r_1=r_2=29$, so by Case \#3 we have $r_1,r_2 \ge R(C_4,K_3,K_4)$.
With $m=1$, $n=2$ and $G_1=G_2=K_4$, we obtain
$$R(C_4,K_4,K_4) \le r_1+r_2+\left\lceil\sqrt{r_1+r_2-1}\right\rceil=
58+\left\lceil\sqrt{57}\right\rceil=66.$$

\item[\bf \#5.]
It is known that $R(C_4,K_3,K_3)=17$ \cite{ExRe}.
Let $r_i=17$ and $G_i=K_3$, for $1 \le i \le 3$.
With $m=1$ and $n=3$, we have
$$R(C_4,K_3,K_3,K_3)\le r_1+r_2+r_3-1+\left\lceil\sqrt{r_1+r_2+r_3-2}\right\rceil=
50+\sqrt{49}=57.$$

\item[\bf \#6.]
It is known that $R(C_4,C_4,K_4)\le 21$ \cite{LidP} and
$R(C_4,C_4,K_3,K_3)\le 36$ \cite{XuR2}. Let $r_1=21$ and $r_2=36$,
so that $r_1 \ge R(C_4,C_4,K_4)$ and $r_2 \ge R(C_4,C_4,K_3,K_3)$.
With $m=2$, $n=2$, $G_1=K_3$ and $G_2=K_4$, we have
$$R(C_4,C_4,K_3,K_4)\le r_1+r_2+2
+\left\lceil 2\sqrt{9/4+r_1+r_2-2}\right\rceil =
59+\left\lceil\sqrt{9+220}\right\rceil=75.$$

\item[\bf \#7.]
Let $r_1=r_2=75$, so by Case \#6 we have $r_1,r_2 \ge R(C_4,C_4,K_3,K_4)$.
With $m=2$, $n=2$ and $G_1=G_2=K_4$, we have
$$R(C_4,C_4,K_4,K_4)\le r_1+r_2+2
+\left\lceil2\sqrt{9/4+r_1+r_2-2}\right\rceil =
152+\left\lceil\sqrt{1+600}\right\rceil=177.$$
\end{itemize}
\end{proof}

\section{Stars and Books}

We start this section with a classical result obtained by Parsons in 1975.
\begin{lemma} \label{Par}
\cite{Par3}
For $k\ge 2$, we have $R(C_4,K_{1,k})\le k+\left\lceil\sqrt{k}\right\rceil+1$.
\end{lemma}
\begin{proof}
The original proof was presented by Parsons, but we note that the same result
is implied by our Theorem \ref{MT} using $m=n=1$ and $r_1=R(C_4,kK_1)=k$.
\end{proof}

\medskip
The next corollary puts together Theorem \ref{MT} and Lemma \ref{Par}.
\begin{corollary} \label{book}
For $k\ge 2$, we have $$R(C_4,B_k)\le R(C_4,K_{1,k})+\left\lceil\sqrt{R(C_4,K_{1,k})}\right\rceil+1\le k+\left\lceil \sqrt{k+\left\lceil\sqrt{k}\right\rceil+1}\right\rceil +\left\lceil \sqrt{k}\right\rceil+2.$$
\end{corollary}
\begin{proof}
Since $B_k=K_1+K_{1,k}$, using Theorem \ref{MT} with $m=n=1$ and $r_1=k+\left\lceil\sqrt{k}\right\rceil+1$ gives the first inequality. The second inequality is obtained by Lemma \ref{Par}.
\end{proof}

Note that for $k=q^2-q+1$ we have $\left\lceil\sqrt{k}\right\rceil=q$. Our
result in Corollary \ref{book}, which holds for all integers $k \ge 2$, generalizes
a result by Faudree, Rousseau and Sheehan \cite{FRS7}. In particular, the Lemma in
section 2 of \cite{FRS7} implies that $R(C_4,B_{17})\le 29$, while our Corollary \ref{book}
using Parson's \cite{Par3} result $R(C_4,K_{1,17})=22$ gives a better bound, namely
$R(C_4,B_{17})\le 28$.

Our last corollary about multicolor Ramsey numbers of $C_4$'s versus stars is also a consequence of Theorem \ref{MT}.
\begin{corollary} \label{c4stars}
Let $m,n,k_1,\ldots,k_n\ge 1$, such that $m+\sum_{i=1}^n k_i\ge n+2$. Then
$$R(\underbrace{C_4,\ldots,C_4}_m,K_{1,k_1},\ldots,K_{1,k_n})\le 1+\sum_{i=1}^nk_i-n+\dfrac{m^2+m}{2}+\left\lceil m\sqrt{\dfrac{(m+1)^2}{4}+\sum_{i=1}^n k_i-n}\right\rceil.$$
\end{corollary}
\begin{proof}
Apply Theorem \ref{MT} with $$r_i=R(\underbrace{C_4,\ldots,C_4}_m,K_{1,k_1},\ldots,K_{1,k_{i-1}},k_iK_1,K_{1,k_{i+1}},\ldots,K_{1,k_n}),$$

\noindent
and observe that $r_i=k_i$. The result follows.
\end{proof}

\medskip
We note that if we consider Corollary \ref{c4stars} with $n=1$, then it reduces to a result
obtained in \cite{ZhaCC4}, which states that for $k,m\ge 1$, with $m+k>3$, it holds that
$$R(\underbrace{C_4,\ldots,C_4}_m,K_{1,k})\le k+\dfrac{m^2+m}{2}+\left\lceil m\sqrt{k+\dfrac{m^2+2m-3}{4}}\right\rceil.$$


\begin{thebibliography}{99}
\bibitem{Bev}
\newblock D. Bevan.
\newblock Personal communication to S. Radziszowski, 2002.

\bibitem{Chu2}
\newblock F.R.K. Chung.
\newblock On Triangular and Cyclic Ramsey Numbers with $k$ Colors.
\newblock {\em Graphs and Combinatorics} (R. Bari and F. Harary eds.), Springer LNM406, Berlin, 236-241, 1974.

\bibitem{DyDz1}
\newblock J. Dybizba\'nski and T. Dzido.
\newblock On Some Ramsey Numbers for Quadrilaterals.
\newblock {\em Electronic Journal of Combinatorics}, 18(1), \#P154, 12 pages, 2011.

\bibitem{Eng}
\newblock A. Engel.
\newblock Problem-Solving Strategies.
\newblock Springer-Verlag New York, Inc, 1998.

\bibitem{ERF}
\newblock P. Erd\H{o}s, A. R\'enyi and V.T. S\'os.
\newblock On a Problem of Graph Theory.
\newblock {\em Studia Scientiarium Mathematicarum Hungarica}, 1, 215-235, 1966.

\bibitem{ExRe}
\newblock G. Exoo and D.F. Reynolds.
\newblock Ramsey Numbers Based on $C_5$-Decompositions,
\newblock {\em Discrete Mathematics}, 71, 119-127, 1988.

\bibitem{FRS7}
\newblock R.J. Faudree, C.C. Rousseau and J. Sheehan.
\newblock More from the Good Book.
\newblock Proceedings of the Ninth Southeastern Conference on Combinatorics, Graph Theory, and Computing.
\newblock {\em Utilitas Mathematica Publ., Congressus Numerantium XXI}, 289-299, 1978.

\bibitem{GG}
\newblock R.E. Greenwood and A.M. Gleason.
\newblock Combinatorial Relations and Chromatic Graphs.
\newblock {Canadian Journal of Mathematics}, 7, 1-7, 1955.

\bibitem{Ir}
\newblock R.W. Irving.
\newblock Generalised Ramsey Numbers for Small Graphs.
\newblock {\em Discrete Mathematics}, 9, 251-264, 1974.

\bibitem{LaLR}
\newblock A. Lange, I. Livinsky and S.P. Radziszowski.
\newblock Computation of  the Ramsey  Numbers $R(C_4,K_9)$ and $R(C_4,K_{10})$.
\newblock {\em Journal of Combinatorial Mathematics and Combinatorial Computing}, 97, 139-154, 2016.

\bibitem{LidP}
\newblock B. Lidick\'y and F. Pfender.
\newblock Semidefinite Programming and Ramsey Numbers.
\newblock Preprint, arXiv, http://arxiv.org/abs/1704.03592 (2017). Revised version (2020).

\bibitem{Par3}
\newblock T.D. Parsons.
\newblock Ramsey Graphs and Block Designs, I.
\newblock {\em Transactions of the American Mathematical Society}, 209, 33-44, 1975.

\bibitem{R}
\newblock S. Radziszowski.
\newblock Small Ramsey Numbers.
\newblock {\em The Electronic Journal of Combinatorics}, Dynamic Surveys 1, 2021.

\bibitem{VO}
\newblock S. Van Overberghe.
\newblock Algorithms for Computing Ramsey Numbers.
\newblock {\em MS Thesis in Mathematics, Ghent University, Belgium}, 2020.
\newblock Constructions at https://github.com/Steven-VO/circulant-Ramsey.

\bibitem{XuR2}
\newblock Xiaodong Xu and S.P. Radziszowski.
\newblock $28\le R(C_4,C_4,C_3,C_3)\le 36$.
\newblock {\em Utilitas Mathematica}, 79, 253-257, 2009.

\bibitem{XSR1}
\newblock Xiaodong Xu, Zehui Shao and S.P. Radziszowski.
\newblock Bounds on Some Ramsey Numbers Involving Quadrilateral.
\newblock {\em Ars Combinatoria}, 90, 33-344, 2009.

\bibitem{ZhaCC4}
\newblock Xuemei Zhang, Yaojun Chen and T.C. Edwin Cheng.
\newblock On Three Color Ramsey Numbers for $R(C_4,C_4,K_{1,n})$.
\newblock {\em Discrete Mathematics}, 342, 285-291, 2019.
\end{thebibliography}
\end{document}